\documentclass[12pt,reqno]{amsart}
\usepackage[]{hyperref}
\usepackage{amsthm,amsfonts,amsmath,amssymb,enumerate}

\newtheorem{thm}{Theorem}[section]
\newtheorem{cor}[thm]{Corollary}

\theoremstyle{definition}
\newtheorem{defn}[thm]{Definition}
\newtheorem{rem}[thm]{Remark}
\newtheorem{exam}[thm]{Example}

\numberwithin{equation}{section}
\usepackage{graphicx}
\usepackage[a4paper,
left=2.5cm,
right=2.5cm,
top=3cm,
bottom=3 cm,
]{geometry}

\begin{document}

\title[Stability of Euler-Lagrange type cubic functional equations]{Stability of Euler-Lagrange type cubic functional equations in quasi-Banach spaces}

\author{Wutiphol Sintunavarat}

\address{Department of Mathematics and Statistics, Faculty of Science and Technology, Thammasat University Rangsit Center, 12121 Pathumthani, Thailand}

\email{wutiphol@mathstat.sci.tu.ac.th}

\author[]{Nguyen Van Dung}

\address{Faculty of Mathematics Teacher Education, Dong Thap University, Cao Lanh City, Dong Thap Province, Vietnam}

\email{nvdung@dthu.edu.vn}

\author[]{Anurak Thanyacharoen}

\address{Department of Mathematics and Statistics, Faculty of Science and Technology, Thammasart University, Phatumthani 12121, Thailand}


\email{poiuy@yahoo.com}

\subjclass[2000]{Primary 39B32; Secondary 30D05}%
\keywords{quasi-normed; cubic functional equation; Euler-Lagrange type functional equation}

\begin{abstract}
In this paper, we study the generalized Hyers-Ulam stability of Euler-Lagrange type cubic functional equation of the form
\begin{align*}
2mf(x + my) + 2f(mx - y) = (m^3 + m)[f(x+ y) + f(x - y)] + 2(m^4 - 1)f(y)
\end{align*}
for all $x,y \in X$, where $m$ is a fixed scalar such that $m \neq 0,1$, and $f$ is a map from a  quasi-normed space $X$ to a quasi-Banach space $Y$ over the same field with $X$ by applying the alternative fixed point theorem.
\end{abstract}
\maketitle

\section{Introduction and preliminaries}
A functional equation of the form 
 \begin{equation} \label{i-2}
 f(2x + y) + f (2x - y) = 2f (x + y) + 2f (x - y) + 12f (x)
 \end{equation}
was introduced by Jun and Kim \cite{JK2002} which is said to be a \textit{cubic functional} equation and every solution of \eqref{i-2} is called a \textit{cubic function}.
One of the solutions of \eqref{i-2} is the function $f$ defined by $f(x)=cx^3$ for all $x \in \mathbb{R}$, where $c$ is an arbitrary real constant. Jun and Kim \cite{JK2002} also found 
the general solution of \eqref{i-2} on real vector spaces and investigated its Hyers-Ulam stability problem on real Banach spaces.

A general form of the functional equation \eqref{i-2} was showed by Jun \emph{et al.} \cite{JKC2005} given by
 \begin{equation} \label{i-2-1}
f(ax + y) + f(ax - y) = af(x + y) + af(x - y) + 2a(a^2 - 1)f(x) 
 \end{equation}
for a fixed integer a with $a \neq 0,±1$. 

In 2007, Jun and Kim \cite{JK2007} investigated the generalized Hyers-Ulam stability problem for Euler-Lagrange type
cubic functional equation of the form
 \begin{align} \label{1-0-1}
 f (ax + y) + f (x + ay) = (a + 1)(a - 1)^2[ f(x) + f (y)]+ a(a + 1)f (x + y) 
 \end{align}
in quasi-Banach spaces, where $a$ is a fixed integer with $a \neq 0,±1$.

In 2009, Najati and Moradlou \cite{NM2009} solved the general solution and considered the generalized Hyers-Ulam
stability problem for Euler-Lagrange type cubic functional equation of the form
 \begin{align} \label{1-00}
 2mf(x + my) + 2f(mx - y) = (m^3 + m)[f(x+ y) + f(x - y)] + 2(m^4 - 1)f(y)
 \end{align}
for all $x,y \in X$, where $m$ is a fixed integer such that $m \neq 0,\pm 1$, in Banach spaces and left Banach modules over a unital Banach $\ast$-algebra.
Also, the stability result for the functional equation \eqref{1-00} was investigated by Saadati \emph{et al.} \cite{SCR2015} in the $\mathcal{L}$-fuzzy normed space and the non-Archimedean $\mathcal{L}$-fuzzy normed space. 

Recall that the quasi-Banach space is an interesting generalization of a Banach space~\cite{NK2003}. The stability of functional equations in quasi-Banach spaces was first studied by Najati and Moghimi~\cite{NM2008} and Najati and Eskandani~\cite{NE2008}. Recently, some results on stability of functional equations in quasi-Banach spaces were proved~\cite{DH2018-7}, \cite{XRX2012}. The key difference between a quasi-norm and a norm is that the modulus of concavity of a quasi-norm is greater than or equal to $1$, while that of a norm is equal to $1$. This causes the quasi-norm to be not continuous in general, while a norm is always continuous. Moreover, a quasi-normed space is not normable in general.

In this paper, we investigate the generalized Hyers-Ulam stability of Euler-Lagrange type cubic functional equation of the form
 \begin{align} \label{1-1}
 2mf(x + my) + 2f(mx - y) = (m^3 + m)[f(x+ y) + f(x - y)] + 2(m^4 - 1)f(y)
 \end{align}
for all $x,y \in X$, where $m$ is a fixed scalar such that $m \neq 0,1$ and $f$ maps from a quasi-normed space $X$ to a quasi-normed space $Y$ over the same field with $X$ by applying the alternative fixed point theorem.

Next, we introduce important definitions and some related results.
%
%

\begin{defn}[\cite{NK2003}]
Let $X$ be a vector space over the field $\mathbb{K}$ ($\mathbb{R}$ or $\mathbb{C}$).
A function $\|\cdot\|:X\rightarrow \mathbb{R}_+$ is called a \textit{quasi-norm} if it satisfies the following conditions:
\begin{enumerate}
 \item $\left\|x\right\| = 0$ if and only if $x = 0$;
 \item $\left\|rx\right\|= \left|r\right|\left\|x\right\|$ for all $r \in \mathbb{K}$ and all $x \in X$;
 \item there is a constant $\kappa \geq 1$ such that $\left\|x+y\right\| \leq \kappa \left(\left\|x\right\|+\left\|y\right\|\right)$ for all $x, y \in X$.
\end{enumerate}
Also, $(X,\left\|\cdot\right\|,\kappa)$ is called a \textit{quasi-normed space}. The smallest possible $\kappa$ is called the modulus of concavity of $\left\|\cdot\right\|$. 
\end{defn}

\begin{defn}[\cite{NK2003}]
The sequence $\{x_n\}$ in a quasi-normed space $(X,\left\|\cdot\right\|,\kappa)$ is \textit{convergent} to a point $x$ in $X$ if $\underset{n\rightarrow \infty}{\lim}\|x_n-x\|=0$.
If $\underset{n, m \rightarrow \infty}{\lim}\|x_n-x_m\|=0$, the sequence $\{x_n\}$ in $X$ is called \textit{a Cauchy sequence}. The space $(X,\left\|\cdot\right\|,\kappa)$ is called \textit{quasi-Banach space} if every Cauchy sequence is convergent.
\end{defn}

\begin{defn}[\cite{NK2003}]
The quasi-norm $\left\|\cdot\right\|$ is called a \textit{$p$-norm}
if there exists a number $p$ with $0 < p \leq 1$ such that
$$\left\|x+y\right\|^p \leq \left\|x\right\|^p+\left\|y\right\|^p$$
for all $x, y \in X$. Also, $(X,\left\|\cdot\right\|,\kappa)$ is called \textit{$p$-Banach space} if $\left\|\cdot\right\|$ is a $p$-norm and $X$ is a quasi-Banach space.
\end{defn}

\begin{defn}(\cite{SC1998})
 Let $X$ be a nonempty set, $\kappa \geq 1$ and $d : X \times X \rightarrow [0,\infty) $
 be a function such that for all $x, y, z \in X$,
 \begin{enumerate}
 \item $d(x,y) = 0$ if and only if $x = y$;
 \item $d(x,y) = d(y,x)$;
 \item $d(x,y) \leq \kappa \left(d(x,z)+d(z,y)\right)$.
 \end{enumerate}
Then $d$ is called a \textit{$b$-metric} on $X$ and $(X, d, \kappa)$ is called a \textit{$b$-metric space}.
\end{defn}

\begin{defn}(\cite{SC1998})
The sequence $\{x_n\}$ in a $b$-metric space $(X, d, \kappa)$ is \textit{convergent} to a point $x$ in $X$ if $\underset{n\rightarrow \infty}{\lim}(x_n,x)=0$.
If $\underset{n, m \rightarrow \infty}{\lim}(x_n,x_m)=0$, the sequence $\{x_n\}$ in $X$ is called \textit{Cauchy}. The space $(X, d, \kappa)$ is called \textit{complete} if every Cauchy sequence is convergent.
\end{defn}

The following results involving a quasi-normed space with some $p$-norm are important tools for proving main results in this paper. 

\begin{thm}[\cite{LM2006}] \label{TT1}
Let $(Y,\left\|\cdot\right\|_Y,\kappa_Y)$ be a quasi-normed space, $p = \log_{2\kappa_Y}2$,
and
 \begin{equation}
\left|\left\|x\right\|\right|_Y = \inf\left\{\left(\sum_{i=1}^{n}\left\|x_i\right\|_Y^p\right)^{\frac{1}{p}}: x= \sum_{i=1}^{n}x_i, x_i\in Y,n\geq 1 \right\}
 \end{equation}
for all $x \in Y$. Then $\left\|\cdot\right\|_Y$ is a quasi-norm on $Y$ satisfying
 $$\left|\left\|x+y\right\|\right|_Y^p \leq \left|\left\|x\right\|\right|_Y^p+\left|\left\|y\right\|\right|_Y^p $$
 and 
\begin{equation} \label{SDA2018-01}
\frac{1}{2\kappa_Y}\left\|x\right\|_Y \leq \left|\left\|x\right\|\right|_Y \leq \left\|x\right\|_Y
\end{equation}
for all $x \in Y$. In particular, the quasi-norm $\left|\left\|\cdot\right\|\right|_Y$ is a $p$-norm, and if $\left\|\cdot\right\|_Y$ is a norm,  then $p = 1$ and $\left|\left\|\cdot\right\|\right|_Y=\left\|\cdot\right\|_Y$.
\end{thm}

Now, we recall the following fixed point theorem in complete generalized metric space.  

\begin{thm}[\cite{DM1968}, Theorem on page~306]\label{F1} Let $(X,d)$ be a complete generalized metric space and let
$T : X \rightarrow X$ be a map satisfying $d(Tx,Ty) \le Ld(x,y)$ for all $x,y \in X$ and for some $0 \le L < 1$. Then for each $x \in X$, we have
\begin{enumerate}
 \item either $d(T^nx,T^{n+1}x) = \infty$
for all $n \in \mathbb{N}$,
\item or the following assertions hold:
\begin{enumerate}
 \item \label{AD2018-06} $\lim\limits_{n \to \infty}T^nx = x^*$ where $x^*$ is a fixed point of $T$;
 \item $d(x,x^*) \le \frac{1}{1-L}d(x,Tx)$.
\end{enumerate} 
 \end{enumerate}
\end{thm} 

\begin{rem}
The conclusion ~\eqref{AD2018-06} was not stated in the original version but it is reduced easily from the following inequality $$ d(x,x^*) = d(x,Tx^*) \le d(x,Tx) + d(Tx,Tx^*) \le d(x,Tx) + L d(x,x^*)$$
for all $x\in X$. 
\end{rem}

\begin{thm}[\cite{pS2009}]\label{F2}
 Let $(X,D,\kappa)$ be a $b$-metric space, $p=\log_{2\kappa}2$ and 
 \begin{equation*}
 d(x,y)=\inf\left\{\sum_{i =1}^{n}D^p(x_{i-1},x_{i}):x_0=x,x_1,\ldots,x_{n-1},x_{n}=y \in X, n \in \mathbb{N} \right\}
 \end{equation*}
for all $x, y \in X$. Then $d$ is a metric on $X$ satisfying $\frac{1}{4}D^p \leq d \leq D^p$. In particular, if $D$ is a metric
then $d = D$.
\end{thm}

\section{Main results}
First, we construct a generalized metric from a given generalized $b$-metric as follows.

\begin{thm}\label{FF2}
Let $(X,D,\kappa)$ be a generalized $b$-metric space, $0 < p \leq 1$ satisfying $p =\log_{2\kappa}2$, and 
 \begin{equation} \label{d1}
 \delta(x,y)=\inf\left\{\sum_{i =1}^{n}D^p(x_{i-1},x_{i}):x_0=x,x_1,\ldots,x_{n-1},x_{n}=y \in X, n \in \mathbb{N} \right\}
 \end{equation}
for all $x, y \in X$. Then $\delta$ is a generalized metric on $X$ satisfying 
 \begin{equation} \label{T-111}
 \frac{1}{4}D^p \leq \delta \leq D^p.
 \end{equation}
In particular, if $D$ is a generalized metric, then $\delta = D$.
\end{thm}

\begin{proof} For all $x,y \in X$, we find that $0 \le \delta(x,y) \le \infty$; $\delta(x,y) = \delta(y,x)$ and $\delta(x,y) \le D^p(x,y)$. 
 
We will show that for all $x_0=x,x_1,\dots,x_{n}=y \in X$, 
 \begin{equation}\label{T-1}
 D^p(x,y) \le 2 \left[ D^p(x_0,x_1) + 2 \sum_{j=1}^{n-2}D^p(x_{j},x_{j+1}) + D^p(x_{n-1},x_{n}) \right]
 \end{equation}
by using the strong mathematical induction on $n \geq 2$. 
For the case $n=2$, we need to prove that
 \begin{equation}\label{T-2}
 D^p(x,y) \le 2 \left[ D^p(x,x_1) + D^p(x_{1},y) \right]
 \end{equation}
for all $x,x_1, y \in X$. For each $x,y, z \in X$, we have 
\begin{equation}\label{T-3}
 D(x,y) \le \kappa\left[ D(x,z) + D(z,y) \right]
 \end{equation} 
 and so 
 \begin{equation}\label{T-4}
 D(x,y) \le 2\kappa\max \left\{ D(x,z), D(z,y) \right\}.
 \end{equation} Therefore
 \begin{eqnarray}\label{T-6}
 D^p(x,y) &\le& (2\kappa)^p \max\left\{ D^p(x,x_1), D^p(x_{1},y) \right\} \nonumber\\
 &=& 2 \max\left\{ D^p(x,x_1), D^p(x_{1},y) \right\} \nonumber\\
 & \le & 2 [D^p(x,x_1) + D^p(x_{1},y) ].
 \end{eqnarray} 
 
Next, we suppose that \eqref{T-1} holds for all $n\leq k \in \mathbb{N}-\{1\}$. Let $x=x_0,x_1,\ldots,x_k,x_{k+1}=y \in X$. We need to show that 
 \begin{equation}\label{AD2018-01}
D^p(x,y) \le 2 \left[ D^p(x_0,x_1) + 2 \sum_{j=1}^{k-1}D^p(x_{j},x_{j+1}) + D^p(x_{k},x_{k+1}) \right].
\end{equation}
We find that $D^p(x,y) \le 2D^p(x_0,y)$. So there exists $$ m =\max\big\{0, 1, \ldots, k:
 D^p(x,y) \le 2D^p(x_m,y)\big\}.$$ If $m =k$ then $D^p(x,y) \le 2D^p(x_k,y)$. Therefore~\eqref{AD2018-01} holds. So we may assume $0 \le m \le k-1$. It follows that $ D^p(x,y) >2D^p(x_{m+1},y)$. 
By~\eqref{T-3} we have $$D^p(x,y) \le 2 \max\left\{ D^p(x,x_{m+1}), D^p(x_{m+1},y) \right\}.$$ It implies that
 \begin{equation}\label{T-8}
 D^p(x,y) \le 2D^p(x,x_{m+1}).
 \end{equation} 
 If $m =0$, then $D^p(x,y) \le 2D^p(x,x_{1})$. Therefore~\eqref{AD2018-01} holds.
So we may assume that $1 \le m \le k-1$. We find that
 \begin{align}\label{T-9} 
 D^p(x,y) &\le 2\min\left\{D^p(x,x_{m+1}), D^p(x_m,y) \right\} \nonumber\\
 &\le D^p(x,x_{m+1})+ D^p(x_m,y).
 \end{align}
Now, applying the induction hypothesis for $$x_0=x,x_1,\ldots,x_m,x_{m+1} \text{ and } x_m,x_{m+1},\ldots,x_{k},x_{k+1}=y$$ yields
 \begin{eqnarray*} 
 D^p(x,y) &\le& D^p(x,x_{m+1})+ D^p(x_m,y)\\
 &\le & 2 \left[ D^p(x_0,x_1) + 2 \sum_{j=1}^{m-1}D^p(x_{j},x_{j+1}) + D^p(x_{m},x_{m+1}) \right]\\ 
 && +2 \left[ D^p(x_m,x_{m+1}) + 2 \sum_{j=m+1}^{k-1}D^p(x_{j},x_{j+1}) + D^p(x_{k},x_{k+1}) \right]\\
 &= & 2 \left[ D^p(x_0,x_1) + 2 \sum_{j=1}^{k-1}D^p(x_{j},x_{j+1}) + D^p(x_{k},x_{k+1}) \right].
 \end{eqnarray*} Then~\eqref{AD2018-01} holds
which complete the proof by induction on $n$ of \eqref{T-1}. From \eqref{d1} we have  
\begin{equation} 
D^p(x,y) \le 4 \left[ D^p(x_0,x_1) + \sum_{j=1}^{n-2}D^p(x_{j},x_{j+1}) + D^p(x_{n-1},x_{n}) \right] = 4 \sum_{j=1}^{n}D^p(x_{j-1},x_{j}) .
\end{equation} Then for all $x,y,z \in X$ we obtain 
 \begin{align}\label{T-10} 
 \frac{D^p(x,y)}{4} \leq \delta(x,y) \leq D^p(x,y).
 \end{align}
 
From \eqref{T-10} we find that $\delta(x,y) = 0 $ if and only if $x=y$. We will show the triangle inequality of $\delta$, that is,
\begin{equation} \label{AD2018-02}
\delta(x,y) \le \delta(x,z) + \delta(z,y).
\end{equation}
For $x,y,z \in X$ we consider the following cases.

\textbf{Case 1.} $\delta(x,y)= \infty$. Suppose that $\delta(x,z) < \infty$ and $\delta(z,y) < \infty$. Then there exist $x_0^*=x,x_1^*,\ldots,x_s^*=z$, $x_{s+1}^*,\ldots,x_{s+r}^*=y$ such that 
 \begin{align*} 
 \sum_{j=1}^{s}D^p(x_{j-1}^*,x_{j}^*)< \delta(x,z)+1 \\ 
 \sum_{j=1}^{r}D^p(x_{s+j-1}^*,x_{s+j}^*)< \delta(z,y)+1.
 \end{align*} We find that
 \begin{eqnarray*} 
 \infty &=& \delta(x,y)\\
 &= &\inf\left\{\sum_{i =1}^{n}D^p(x_{i-1},x_{i}):x_0=x,x_1,\ldots,x_{n-1},x_{n}=y \in X, n \in \mathbb{N} \right\}\\
 &\le& \sum_{i=1}^{s+r}D^p(x_{j-1}^*,x_{j}^*)\nonumber\\
 &\leq& \sum_{j=1}^{s}D^p(x_{j-1},x_{j})+\sum_{j=1}^{r}D^p(z_{j-1},z_{j})\\
 & < & \delta(x,z) + 1 + \delta(z,y) + 1 \\
 &<& \infty 
 \end{eqnarray*}
which is a contradiction. So $\delta(x,z) = \infty$ or $\delta(z,y) = \infty$. This proves that ~\eqref{AD2018-02} holds.

\textbf{Case 2.} $\delta(x,y)< \infty$. If $\delta(x,z) = \infty$ or $\delta(z,y) =\infty$ then ~\eqref{AD2018-02} holds. So we may assume that $\delta(x,z) < \infty$ and $\delta(z,y) < \infty$. Then for each $\epsilon >0$, there exist $x_0^*=x,x_1^*,\ldots,x_{s}^*=z$ and $x_s^*=z,x_{s+1}^*,\ldots,x_{s+r}^*=y$ such that 
 \begin{align*} 
 \sum_{i=1}^{s}D^p(x_{i-1}^*,x_{i}^*)< \delta(x,z)+\frac{\epsilon}{2} \\ 
 \sum_{i=s+1}^{s+r}D^p(x_{i-1}^*,x_{i}^*)< \delta(z,y)+\frac{\epsilon}{2}.
 \end{align*}

We find that
 \begin{eqnarray*} 
\delta(x,y)
 &=& \inf\left\{\sum_{i =1}^{n}D^p(x_{i-1},x_{i}):x_0=x,x_1,\ldots,x_{n-1},x_{n}=y \in X, n \in \mathbb{N} \right\}\\
 & \le & \sum_{i =1}^{s+r}D^p(x_{i-1}^*,x_{i}^*)\nonumber\\
 &=& \sum_{i=1}^{s}D^p(x_{i-1}^*,x_{i}^*) + \sum_{i=s+1}^{s+r}D^p(x_{i-1}^*,x_{i}^*) \nonumber\\
 &<& \delta(x,z)+\frac{\epsilon}{2} + \delta(z,y)+\frac{\epsilon}{2} \nonumber\\
 & =& \delta(x,z)+ \delta(z,y) + \epsilon.
 \end{eqnarray*}
Letting $\epsilon \to 0^+$, we get 
 \begin{align*} 
 \delta(x,y) \leq \delta(x,z)+\delta(z,y).
 \end{align*} This proves that ~\eqref{AD2018-02} holds.
 
Finally, if $D$ is a generalized metric then $p =1$. Then for all $x,y \in X$ and $x_0=x,x_1,\ldots,x_n =y$ we have $D(x,y) \le \sum_{i=1}^{n}D(x_{i-1},x_i)$. 
This implies that $D(x,y) \le \delta(x,y)$. So we have $\delta = D$. 
\end{proof}

Next, we present the following fixed point theorem which is an analogue of Theorem~\ref{F1} in generalized $b$-metric spaces. This result is an important tool to formulate our stability results. 

\begin{thm} \label{F3}
Let $(X,D,\kappa)$ be a complete generalized $b$-metric space and $T : X \rightarrow X$
be a map satisfying $D(Tx,Ty) \le LD(x,y)$ for all $x,y \in X$ and some $0 \le L <1$. Then for each $x \in X$, we have
\begin{enumerate}
 \item either $D(T^nx,T^{n+1}x) = \infty$
for all $n \in \mathbb{N}$, 
\item or the following assertions hold:
\begin{enumerate}
 \item $\lim\limits_{n\to \infty}T^nx=x^*$ where $x^*$ is a fixed point of $T$;
\item there exists $N$ such that for all $n >N$, $$D(T^nx,x^*) \le \big(\frac{4}{1-L}\big)^{\frac{1}{p}}D(T^Nx,T^{N+1}x).$$
 \end{enumerate}
\end{enumerate}
\end{thm}

\begin{proof}
Let $\delta$ be defined by \eqref{d1}. Then by Theorem \ref{FF2}, $\delta$ is a generalized metric on $X$ satisfying $\frac{1}{4}D^p \leq \delta \leq D^p$.
Since $(X,D,\kappa)$ is a complete generalized $b$-metric space, we find that $(X,\delta)$ is a complete generalized metric space.
For all $x,y \in X$, $x_0=x,x_1,\ldots,x_{n}=y \in X$ and $n \in \mathbb{N}$ we have 
 \begin{align*}
 \delta(Tx,Ty) &= \inf\left\{\sum_{i=1}^{n}D^p(y_{i-1},y_{i}):y_0=Tx,y_1,\ldots,y_{n-1},y_{n}=Ty \in X,n \in \mathbb{N} \right\}\\
 &\leq \sum_{i=1}^{n}D^p(Tx_{i-1},Tx_{i})\\
 &\leq L^p\sum_{i=1}^{n}D^p(x_{i-1},x_{i}).
 \end{align*} This implies that
 \begin{align*}
 \delta(Tx,Ty) &\leq L^p\inf\left\{\sum_{i=1}^{n}D^p(x_{i-1},x_{i}):x=x_0,x_1,\ldots,x_{n}=y \in X,n \in \mathbb{N} \right\}\\
 &= L^p\delta(x,y).
 \end{align*} 
Note that $0 \leq L^p < 1$. So applying Theorem \ref{F1} for the map $T$ on the complete generalized metric space $(X,\delta)$, we have
\begin{enumerate}
 \item either $\delta(T^nx,T^{n+1}x) = \infty$ for all $n \in \mathbb{N}$, 
 
\item or the following assertions hold:
\begin{enumerate}
 \item $\lim\limits_{n\to \infty}T^nx=x^*$ in $(X,\delta)$, where $x^*$ is a fixed point of $T$;
 \item $\delta(x,x^*) \le \frac{1}{1-L} \delta(x,Tx)$.
\end{enumerate}
\end{enumerate} 

By~\eqref{T-111} we find that \begin{enumerate}
 \item either $D(T^nx,T^{n+1}x) = \infty$ for all $n \in \mathbb{N}$, or
 \item $\lim\limits_{n\to \infty}T^nx=x^*$ in $(X,D,\kappa)$, where $x^*$ is a fixed point of $T$.
\end{enumerate} Moreover by~\eqref{T-111} we have for all $n >N$, $$ D(x,x^*) \le 4^{\frac{1}{p}}\delta^{\frac{1}{p}}(x,x^*) \le \big(\frac{4}{1-L}\big)^{\frac{1}{p}} \delta^{\frac{1}{p}}(x,Tx) \le \big(\frac{4}{1-L}\big)^{\frac{1}{p}} D(x,Tx) .$$
\end{proof}

Next, we will investigate the stability of an Euler-Lagrange type cubic functional equation \eqref{1-1} on quasi-normed spaces by applying Theorem~\ref{F3} with the following remark.

\begin{rem} \label{SDT2018-01} 
Let $X$ and $Y$ be two vector spaces over the same field. If $f: X \to Y$ satisfies~\eqref{1-1}, then $f(0)=0$. 
Moreover, by choosing $y =0$ in~\eqref{1-1}, we get 
	$$f(mx) = m^3 f(x)$$
	for all $x\in X$.
\end{rem}

\begin{thm}\label{thm1}
Let $(X,\|.\|_X,\kappa_X)$ be a quasi-normed space, $(Y,\|.\|_Y,\kappa_Y)$ be a quasi-Banach space over the same field with $X$, $\phi : X^2 \rightarrow [0,\infty)$ be a function and $f: X \to Y$ be a map with $f(0) = 0$. Suppose that the following conditions hold:  
\begin{enumerate}
 \item there are $0 \leq L < 1$ and $m$ is a scalar with $m \ne 0,1$
 \begin{align}\label{2-1}
 \phi(mx,my) \leq L|m|^3\phi\left(x,y\right);
 \end{align} 
 for all $x,y \in X$.
\item  
 \begin{eqnarray}\label{2-2}
 &&\left\|2mf(x + my) + 2f(mx - y) - (m^3 + m)[f(x+ y) + f(x - y)]\right.\left.- 2(m^4 - 1)f(y)\right\|_Y \nonumber\\
 &\leq & \phi(x,y).
 \end{eqnarray}
for all $x,y \in X$
\end{enumerate}
Then there exists a unique map $q : X \rightarrow Y$ satisfying \eqref{1-1} and  
\begin{align}\label{2-3}
 \left\|f(x)-q(x)\right\|_Y \leq& \left(\frac{4}{1-L^p}\right)^{\frac{1}{p}}\frac{1}{2\left|m\right|^{3}}\phi(x,0)
 \end{align}
for all x $\in X$ with $p =\log_{2\kappa_Y}2$.
\end{thm}
\begin{proof} 
Let $S = \{g : X\rightarrow Y\}$. Define a function $d:S \times S \rightarrow [0,\infty)$ as follows
 \begin{align*}
 d(g,h) =\inf\{c \ge 0: \left\|g(x)-h(x)\right\|_Y\leq c\phi(x,0) \text{ for all } x\in X\}
 \end{align*} 
for all $g,h \in S$, where $\inf \emptyset = \infty$. 
First, we will show that $d$ is a generalized $b$-metric. 
Let $g,h,u \in S$. It is easy to see that  $d(g,h) = d(h,g)$. Now, if $g=h$ then $d(g,h) = 0$. Note that $\|g(x)-h(x)\|_Y \le d(g,h) \phi(x,0)$ for all $x \in X$. If $d(g,h) =0$, then $\|g(x)-h(x)\|_Y =0$, that is, $g(x) = h(x)$ for all $x \in X$. Then $g = h$. Here, we will claim the last property of a generalized $b$-metric.
For each $x\in X$, we have 
$$ \|g(x) - u(x)\|_Y \le d(g,u) \phi(x,0) \text{ and } \|u(x) - h(x)\|_Y \le d(u,h)\phi(x,0).$$ 
It follows that for all $x \in X$, \begin{eqnarray}
\|g(x) -h(x)\|_Y &\le & \kappa_Y\big(\|g(x) -u(x)\|_Y + \|u(x)-h(x)\|_Y \big) \nonumber\\
& \le & \kappa_Y \big(d(g,u)+ d(u,h)
\big) \phi(x,0).
\end{eqnarray} So we have 
$$d(g,h) \le \kappa_Y \big(d(g,u)+ d(u,h)\big).$$ 
Therefore, $d$ is a generalized $b$-metric with the coefficient $
\kappa_Y$ on $S$. 

Next, we will show that $(S,d,\kappa_Y)$ is complete. Let $\{f_n\}$ be a Cauchy sequence in $(S,d,\kappa_Y)$. Then we have $\lim\limits_{n,m \to \infty}d(f_n,f_m) =0$. Note that for all $x\in X$, we have \begin{equation} \label{AD2018-03}
\|f_n(x) -f_m(x)\|_Y \le d(f_n,f_m)\phi(x,0).
\end{equation} Then $\lim\limits_{n,m \to \infty}\|f_n(x)-f_m(x)\|_Y=0$. It implies that $\{f_n(x)\}$ is a Cauchy sequence in $(Y,\|.\|_Y,\kappa_Y)$. Since $(Y,\|.\|_Y,\kappa_Y)$ is quasi-Banach, there exists $\lim\limits_{n\to \infty}f_n(x) = y$ in $(Y,\|.\|_Y,\kappa_Y)$. Put $g(x) = y$, we have the map $g:X \to Y$. We will show that $\lim\limits_{n \to \infty} f_n = g$ in $(S,d,\kappa_Y)$. Indeed, for each $\epsilon>0$ there exists $n_0$ such that $d(f_n,f_m) < \epsilon$ for all $n, m \ge n_0$. So from ~\eqref{AD2018-03}, for all $x \in X$ and $n,m \ge n_0$ we have 
\begin{eqnarray} \label{AD2018-04}
\|f_n(x) -f_m(x)\|_Y & \le & \epsilon \phi(x,0).
\end{eqnarray} Letting $m \to \infty$ in~\eqref{AD2018-04} we get for all $x\in X$ and $n \ge n_0$, $$\|f_n(x) -g(x)\|_Y \le \epsilon \phi(x,0).$$ This implies that $d(f_n,g) \le \epsilon$ for all $n \ge n_0$. So $\lim\limits_{n\to \infty} f_n =g$ in $(S,d,\kappa_Y)$. Then $(S,d,\kappa_Y)$ is complete.

Next, letting $y=0$ in \eqref{2-2} and using $f(0)=0$, we get 
\begin{eqnarray} \label{2-5}
\phi(x,0) & \ge & \left\|2mf(x) + 2f(mx) - (m^3 + m)[f(x) + f(x )]\right.\left.- 2(m^4 - 1)f(0)\right\|_Y \nonumber\\
&=& \left\|2f(mx) - 2m^3f (x)\right\|_Y
\end{eqnarray} 
for all $x \in X$. It yields that 
 \begin{align}\label{2-6}
 \left\| \frac{f (mx)}{m^3}-f (x)\right\|_Y \leq \frac{\phi(x,0)}{2\left|m\right|^{3}}
 \end{align}
for all $x \in X$. Define a map $T : S \rightarrow S$ by $(Tg)(x) = \frac{g(mx)}{m^3}$ for all $g \in S$ and all $x \in X$.
Let $g,h \in S$. By~\eqref{2-1}, we have
 \begin{eqnarray*}
 \left\|(Jg)(x)-(Jh)(x)\right\| &=& \left\|\frac{g(mx)}{m^3}-\frac{h(mx)}{m^3}\right\|_Y \\
 &=& \frac{1}{\left|m\right|^{3}}\left\|g(mx)-h(mx)\right\|_Y\\
 &\leq& \frac{d(g,h)}{\left|m\right|^{3}}\phi(mx,0)\\
 &\leq& \frac{d(g,h)}{|m|^3}L|m|^3\phi(x,0) \nonumber\\
 &=& L d(g,h) \phi(x,0).
 \end{eqnarray*} So we get
 \begin{align*}
 d(Jg,Jh) &\leq Ld(g,h)
 \end{align*}
for all $g,h \in S$. Note that $0 \le L <1$. By Theorem ~\ref{F3}, for each $g \in S$, we have
\begin{enumerate}
 \item either $d(T^ng,T^{n+1}g) = \infty$
 for all $n \in \mathbb{N}$, 
 \item or the following assertions hold:
\begin{enumerate}
 \item $\lim\limits_{n\to \infty}T^ng=q$ where $q$ is a fixed point of $T$;
 \item $d(g,q) \le \big(\frac{4}{1-L}\big)^{\frac{1}{p}}d(g,Tg)$.
 \end{enumerate}
\end{enumerate} 
From ~\eqref{2-6}, we have for all $x \in X$, $\|Tf(x) - f(x)\|_Y \le \frac{\phi(x,0)}{2|m|^3}$. So $d(Tf,f) \le \frac{1}{2|m|^3} < \infty$. This shows that if we choose $g =f$ then \begin{enumerate}
\item $\lim\limits_{n\to \infty}T^nf=q$.

\item $d(f,q) \le \big(\frac{4}{1-L}\big)^{\frac{1}{p}}d(f,Tf).$
\end{enumerate} So we find that $$d(f,q) \le \big(\frac{4}{1-L}\big)^{\frac{1}{p}}d(f,Tf) \le \big(\frac{4}{1-L}\big)^{\frac{1}{p}} \frac{1}{2|m|^3}. $$ Then for all $x \in X$, $
\left\|f(x)-q(x)\right\|_Y \leq \left(\frac{4}{1-L^p}\right)^{\frac{1}{p}}\frac{1}{2\left|m\right|^{3}}\phi(x,0).$ That is ~\eqref{2-3} holds.

Next, we will prove that $q$ is cubic by using the continuity of $|\|.\||_Y$. For each $x \in X$, note that $(Tf)(x) = \frac{f(mx)}{m^3}$. So $$(T^2f)(x) = \frac{Tf(mx)}{m^3} = \frac{f(m^2x)}{m^{3.2}}, \cdots, (T^nf)(x) = \frac{f(m^nx)}{m^{3n}}. $$
So, by~\eqref{SDA2018-01}, \eqref{2-1} and \eqref{2-2}, we have for all $x,y \in X$,
\begin{eqnarray*}
 &&|\|2mq(x + my) + 2q(mx - y) - (m^3 + m)[q(x+ y) + q(x - y)] - 2(m^4 - 1)q(y)\||^p_Y\\
 & = &|\|2m \lim\limits_{n \to \infty}(T^nf)(x + my) + 2\lim\limits_{n \to \infty}(T^nf)(mx - y) \\
 && - (m^3 + m)[\lim\limits_{n \to \infty}(T^nf)(x+ y) + \lim\limits_{n \to \infty}(T^nf)(x - y)] - 2(m^4 - 1)\lim\limits_{n \to \infty}(T^nf)(y)\||^p_Y\\
 & = &\lim\limits_{n \to \infty} |\|2m (T^nf)(x + my) + 2(T^nf)(mx - y) - (m^3 + m)[(T^nf)(x+ y) + (T^nf)(x - y)] \\
 &&- 2(m^4 - 1)(T^nf)(y)\||^p_Y\\
 & = &\lim\limits_{n \to \infty} |\|2m \frac{f(m^nx +m^{n+1}y)}{m^{3n}} + 2\frac{f(m^{n+1}x -m^ny)}{m^{3n}}\\
 && - (m^3 + m)[\frac{f(m^nx +m^ny)}{m^{3n}} + \frac{f(m^nx -m^ny)}{m^{3n}}] - 2(m^4 - 1)\frac{f(m^ny)}{m^{3n}}\||^p_Y\\
 & = &\lim\limits_{n \to \infty} \frac{1}{|m|^{3np}} |\|2m f(m^nx +m^{n+1}y) + 2f(m^{n+1}x -m^ny) \\
 && - (m^3 + m)[f(m^nx +m^ny) + f(m^nx -m^ny)]- 2(m^4 - 1)f(m^ny) \||^p_Y\\
 & \le &\lim\limits_{n \to \infty} \frac{1}{|m|^{3np}} \|2m f(m^nx +m^{n+1}y) + 2f(m^{n+1}x -m^ny) \\
 && - (m^3 + m)[f(m^nx +m^ny) + f(m^nx -m^ny)]- 2(m^4 - 1)f(m^ny) \|^p_Y\\
 &\leq& \lim\limits_{n \to \infty} \frac{1}{|m|^{3np}}\phi^p(m^nx,m^ny)\\
 &\leq& \lim\limits_{n \to \infty} \frac{L^{np}\left|m\right|^{3np}}{\left|m\right|^{3np}}\phi^p(x,y)\\
 &=& \lim_{n \rightarrow \infty}L^{np}\phi^p(x,y)\\
 & =&0.
 \end{eqnarray*}
This implies that for all $x,y \in X$,
$$2mg(x + my) + 2g(mx - y) - (m^3 + m)[g(x+ y) + g(x - y)] - 2(m^4 - 1)g(y) = 0.$$ 
So $q$ is satisfying \eqref{1-1}. By Lemma~\ref{lem1}, we have $q$ is a cubic map.

Finally, we prove the uniqueness of $q$.
Suppose that $h:X\rightarrow Y$ is also a cubic map satisfying \eqref{2-3}. We need to show that $h = q$. It follows from Remark~\ref{SDT2018-01} that $q(mx) = m^3q(x)$ and $h(mx) = m^3 h(x)$. By using~\eqref{SDA2018-01}, \eqref{2-1} and \eqref{2-3}, for each $n \in \mathbb{N}$, we get 
 \begin{eqnarray*}
 \left|\left\|q(x)-h (x)\right\|\right|^p_Y 
 &=& \left|\left\|\frac{q(m^nx)}{m^{3n}}-\frac{h (m^nx)}{m^{3n}}\right\|\right|^p_Y \\
 &=& \frac{1}{|m|^{3np}}\left|\left\|q(m^nx)-h (m^nx)\right\|\right|^p_Y \\
 &=& \frac{1}{|m|^{3np}}\left|\left\|q(m^nx)- f(m^nx) + f(m^nx) - h (m^nx)\right\|\right|^p_Y \\
& \le & \frac{1}{|m|^{3np}} \big(\left|\left\|q(m^nx)- f(m^nx)\||_Y^p +|\| f(m^nx) - h (m^nx)\right\|\right|^p_Y\big) \\
& \le & \frac{1}{|m|^{3np}} \big(\|q(m^nx)- f(m^nx)\|_Y^p +\| f(m^nx) - h (m^nx)\|^p_Y\big) \\
&\le& \frac{1}{|m|^{3np}} \Big(\frac{4}{1-L^p}\frac{1}{2|m|^{3p}}\phi^p(m^nx,0) + \frac{4}{1-L^p}\frac{1}{2|m|^{3p}}\phi^p(m^nx,0) \Big)\\
& =& \frac{4}{(1-L^p)|m|^{3(n+1)p}}\phi^p(m^nx,0)\\
& \le & \frac{4}{(1-L^p)|m|^{3(n+1)p}}L|m|^{3np}\phi^p(x,0)\\
& = & \frac{4L}{(1-L^p)|m|^{3np}}\phi^p(x,0).
 \end{eqnarray*} Note that $0\leq L<1$ and $p =2\log_{2\kappa}2$. So letting $n\rightarrow \infty$, we get $\left|\left\|q(x)-h (x)\right\|\right|^p_Y = 0$ for all $x\in X$. This proves that $h=q$. 
\end{proof}

\begin{cor}\label{cor1}
Let $X$ be a quasi-normed space, $Y$ be a quasi-Banach space over the same field with $X$ and $f : X \rightarrow Y$ be a map with $f(0) = 0$. Suppose that there are a positive real number $\lambda$, a real number $s<3$ and a real number $m \ne 0, 1$ with $|m|^s-3<1$ such that 
\begin{eqnarray}\label{c1-2-2}
	&&\left\|2mf(x + my) + 2f(mx - y) - (m^3 + m)\right.\left.[f(x+ y) + f(x - y)]- 2(m^4 - 1)f(y)\right\|_Y \notag \\ 
	&\leq& \lambda\left(\left\|x\right\|^s+\left\|y\right\|^s\right).
	\end{eqnarray}
for all $x,y \in X-\{0\}$. 
Then there exists a unique map $q : X \rightarrow Y$ satisfying \eqref{1-1} and 
 \begin{align}\label{c1-2-3}
 \left\|f(x)-q(x)\right\|_Y \leq& \left(\frac{4}{1-|m|^{s-3}}\right)^{\frac{1}{p}}\frac{\lambda}{2\left|m\right|^{3}}\left\|x\right\|^s
 \end{align}
for all $x \in X$, where $p =\log_{2\kappa}2$.
\end{cor}

\begin{proof} 
Define a map $\phi : X^2 \rightarrow [0,\infty)$ by 
\begin{equation*}
 \phi(x,y) :=\left\{
  	\begin{array}{ll}
       0, & \hbox{if $x=0$ or $y=0$;} \\
       \lambda \left(\left\|x\right\|^s+\left\|y\right\|^s\right), & \hbox{otherwise}.
       \end{array}\right.
\end{equation*}
Next, we will show that 
$$\phi(mx,my) \leq L |m|^3\phi(x,y)$$
for all $x,y \in X$, where $L:=|m|^{s-3} \in [0,1)$.
Let $x,y\in X$. If $x=0$ or $y=0$, then 
$$\phi(mx,my) = 0 \leq L |m|^3\phi(x,y).$$ 
If $x\neq 0$ and $y\neq 0$, then we have 
\begin{eqnarray*}
 \phi(mx,my) & = & \lambda\left(\left\|mx\right\|^s+\left\|my\right\|^s\right)\\
 &=& \lambda |m|^s\left(\left\|x\right\|^s+\left\|y\right\|^s\right)\\
 &= & L |m|^3\phi(x,y).
 \end{eqnarray*}
Now, all conditions in Theorem \ref{thm1} hold. Therefore, we obtain this result.
\end{proof}

Next we exemplify that Theorem \ref{thm1} is better than \cite[Theorem 3.1]{NM2009}.

 \begin{exam} \label{SDT2018-02}
 Let $X= Y = L^{\frac{1}{2}}[0, 1]$ with
 $$ L^{\frac{1}{2}}[0, 1]= \left\{x:[0,1]\rightarrow \mathbb{R}:\left|x\right|^{\frac{1}{2}}\text{is Lebesgue integrable}\right\} $$
and
 $$\left\|x\right\|_X=\left\|x\right\|_Y=\left(\int_0^1|x(t)|^{\frac{1}{2}}dt\right)^2 $$
 for all $x\in L^{\frac{1}{2}}[0, 1]$.
 Define $f(x)=x^3+x$ for all $x\in X$ and for some integer $m \ne 0, \pm 1$,
 $$\phi(x,y):= |2m(1-m)|\left(\int_0^1 \left|x(t)+my(t)\right|^{\frac{1}{2}}dt\right)^2.$$
 Then we have
\begin{enumerate}
	\item \label{SDT2018-02-1} $(X,\|.\|_X,\kappa_X)$ and $(Y,\|.\|_Y,\kappa_Y)$ are real quasi-Banach spaces with $\kappa_X = \kappa_Y =2$.
	
	\item \label{SDT2018-02-2} Theorem \ref{thm1} is applicable to $f$ and $\phi$, while \cite[Theorem~3.1]{NM2009} is not applicable to $f$ and $\phi$.
\end{enumerate} 
 \end{exam}

\begin{proof} \eqref{SDT2018-02-1}. See \cite[Example~1]{LM2006}.
	
	\eqref{SDT2018-02-2}. We find that $p = \log_42 = \frac{1}{2}$. For all $x,y \in X$, we also have
	\begin{eqnarray*}
	&& \left\|2mf(x + my) + 2f(mx - y) - (m^3 + m)[f(x+ y) + f(x - y)]- 2(m^4 - 1)f(y)\right\|_Y \\ 
	&=&\Big(\int_0^1 |2m(x(t) + my(t))^3+2m(x(t) + my(t))+2(mx(t) - y(t))^3+2(mx(t) - y(t))\\
	&&- (m^3 + m)[(x(t)+ y(t))^3+(x(t)+ y(t)) + (x(t) - y(t))^3+(x(t)- y(t))]\\
	&&- 2(m^4 - 1)[y(t)^3+y(t)]|^{\frac{1}{2}}dt\Big)^2\\
	&=&\left(\int_0^1\right.\left|2m[x^3(t)+3mx^2(t)y(t)+3m^2x(t)y^2(t)+m^3y^3(t)+x(t)+my(t)]\right.\\
	&&+2[m^3x^3(t)-3m^2x^2(t)y(t)+3mx(t)y^2(t)-y^3(t)+x(t)-my(t)]\\
	&&-(m^3+m)[x^3(t)+3x^2(t)y(t)+3x(t)y^2(t)+y^3(t)+x(t)+y(t)]\\
	&&-(m^3+m)[x^3(t)-3x^2(t)y(t)+3x(t)y^2(t)-y^3(t)+x(t)-y(t)]\\
	&&-2(m^4-1)[y^3(t)+y(t)]\left.\left.\right|^{\frac{1}{2}}dt\right)^2\\
	&=& \left(\int_0^1 \left|(2m-2m^3)x(t)+(2m^2-2m^4)y(t)\right|^{\frac{1}{2}}dt\right)^2\\
	&=& |2m(1-m)|\left(\int_0^1 \left|x(t)+my(t)\right|^{\frac{1}{2}}dt\right)^2\\
	&=& \phi(x,y).
	\end{eqnarray*}
Note that
	\begin{eqnarray*}
	\phi(mx,my) &=& |2m(1-m)|\left(\int_0^1 \left|mx(t)+m^2y(t)\right|^{\frac{1}{2}}dt\right)^2\\
	&=& |m||2m(1-m)|\left(\int_0^1 \left|x(t)+my(t)\right|^{\frac{1}{2}}dt\right)^2\\
	&=& L|m|^3\phi(x,y).
	\end{eqnarray*} Then all assumptions of Theorem \ref{thm1} are satisfied. So Theorem \ref{thm1} is applicable to $f$ and~$\phi$.

Note that $X$ and $Y$ are not normable. So \cite[Theorem~3.1]{NM2009} is not applicable to $f$ and~$\phi$.
\end{proof}

\section*{Acknowledgements}
The first author would like to thank the Thailand Research Fund and Office of the Higher Education Commission under grant no. MRG6180283 for financial support during the preparation of this manuscript.

\bibliographystyle{abbrv} 


\end{document}